\documentclass[12pt]{amsart}
\usepackage{amssymb,amscd,amsmath,amsfonts,epsf,latexsym}
\usepackage{graphicx}
\usepackage{epsfig}

\newtheorem{thm}{Theorem}[section]

\newtheorem{cor}[thm]{Corollary}

\newtheorem{lemma}[thm]{Lemma}

\newtheorem{prop}[thm]{Proposition}

\newtheorem{defn}[thm]{Definition}

\numberwithin{equation}{section}

\newcommand{\R}{{\mathbb{R}}}
\newcommand{\C}{{\mathbb{C}}}

\newcommand{\Z}{{\mathbb{Z}}}

\newcommand{\cG}{{\mathcal{G}}}
\newcommand{\cH}{{\mathcal{H}}}
\newcommand{\cN}{{\mathcal{N}}}
\newcommand{\cR}{{\mathcal{R}}}
\newcommand{\cU}{{\mathcal{U}}}
\newcommand{\cZ}{{\mathcal{Z}}}

\newcommand{\End}{\mathrm{End}}
\newcommand{\Aut}{\mathrm{Aut}}
\newcommand{\Hom}{\mathrm{Hom}}
\newcommand{\SU}{\mathrm{SU}}
\newcommand{\PSU}{\mathrm{PSU}}

\newcommand{\Sp}{\mathrm{Sp}}

\newcommand{\PGL}{\mathrm{PGL}}

\newcommand{\CC}{\mathcal{C}}

\newcommand{\F}{\mathbb{F}}

\newcommand{\Id}{\mathrm{Id}}
\newcommand{\Rep}{\mathrm{Rep}}

\newcommand{\ot}{\otimes}

\begin{document}
\title[Unitary Braid Representations]
{Unitary Braid Representations with Finite Image}

\author{Michael J. Larsen}
\email{mjlarsen@indiana.edu}
\address{Department of Mathematics\\
    Indiana University \\
    Bloomington, IN 47405\\
    U.S.A.}
\thanks{Michael Larsen was partially supported by NSF grant DMS-0354772, Eric Rowell was partially supported by NSA grant H98230-08-1-0020.}

\author{Eric C. Rowell}
\email{rowell@math.tamu.edu}
\address{Department of Mathematics \\
Texas A\&M University \\
College Station, TX 77843 \\
U.S.A.}
\subjclass[2000]{Primary 20F36; Secondary 20C15}
\begin{abstract}
We characterize unitary representations of braid groups $B_n$
of degree linear in $n$ and finite images of such representations of degree
exponential in $n$.
\end{abstract}
\maketitle
\section{Introduction}

In this paper, we prove two loosely connected results about unitary representations of
the braid group $\phi\colon B_n\to U(d)$, when $n$ is sufficiently large and the degree $d$
is not too large compared to $n$.  The original motivation goes back to the work of Jones on
images of Braid groups in Hecke algebra representations $H(q,n)$.  Jones showed
\cite{J1} that when $q=i$, the image of $B_n$ in every
irreducible factor of the Hecke algebra is finite; more explicitly, each such image is
an extension of a symmetric group by a $2$-group.
This is in sharp contrast to the usual behavior
of irreducible factors of Hecke algebra representations, in which the closure of the image of $B_n$
contains all unimodular unitary matrices \cite{FLW}.
Birman and Wajnryb showed \cite{BW} that when $q=e^{2\pi i/6}$, certain factors of
$H(q,n)$ give rise to representations whose images are extensions of symplectic
groups $\Sp(2r,\F_3)$ by $3$-groups, where $n\approx 2r$ (see also \cite{GJ}).
It seems to be known by some experts, though so far as we know it has not appeared in print, that some other factors of $H(e^{2\pi i/6},n)$ give rise to image groups which are extensions of
$\SU(r+1,\F_2)$ by $2$-groups.  Other (extensions of) symplectic groups appear as quotients of the braid group; Wajnryb \cite{W} has found explicit relations exhibiting $\Sp(2r,\F_p)$ as a quotient of $B_{2r+1}$ for all $p$.  We would like to explain in some sense or at least characterize the possibilities for finite images in such representations.  Such a characterization is given in
Theorem~\ref{t:images}.

It appears to be typically the case that the image of $B_n$ in $U(d)$ can be regarded as a linear group, whose rank is comparable to $n$, over a finite field.  We would like to systematically study
all representations of $B_n$ of dimension $O(n)$ over all fields.  Such a study has been initiated
for complex representations of degree $\le n$ by Formanek and his coworkers in \cite{F,FLSV,Sy}.  In Theorem~\ref{t:degrees}, we extend these results to higher multiples of $n$, but only for unitary representations.  For general representations, we have only the very soft result Theorem~\ref{t:rootbound}, which is used to relate $n$ and $r$ in Theorem~\ref{t:images}.

\section{Braid Groups}

In this section we establish some basic facts concerning the braid groups $B_n$ and their \emph{representations} in the general sense of homomorphisms $\phi\colon B_n\to G$ where $G$ is any group.  Propositions~\ref{p:consecutive} and \ref{p:solv} can be found in \cite{F}, but we include full proofs for the reader's convenience.

For each braid group $B_n$ we fix generators $x_1,\ldots,x_{n-1}$ such that
\begin{align}
\label{braid} x_i x_j x_i = x_j x_i x_j \quad&\hbox{if $|i-j|=1$}, \\
\label{comm} x_i x_j =  x_j x_i \quad&\hbox{if $|i-j|\neq1$}.
\end{align}

\begin{defn}
We say a homomorphism $\phi\colon B_n\to G$ is \emph{constant} if
$$\phi(x_1)=\phi(x_2)=\cdots=\phi(x_{n-1}).$$
\end{defn}

\begin{prop}
\label{p:consecutive}
If $\phi\colon B_n\to G$ is a homomorphism and $\phi(x_i)$ commutes with $\phi(x_{i+1})$
for some $i\le n-2$, then $\phi$ is constant.
\end{prop}

\begin{proof}
Applying (\ref{braid}) when $j=i+1$, we get
$$\phi(x_i)^2\phi(x_{i+1}) = \phi(x_{i+1})^2 \phi(x_i),$$
which implies $\phi(x_i) = \phi(x_{i+1})$.
As $x_i$ commutes with $x_{i+2}$, $\phi(x_{i+1})=\phi(x_i)$ commutes with $\phi(x_{i+2})$.
By induction on $i$,
$$\phi(x_i)=\phi(x_{i+1})=\cdots=\phi(x_{n-1}).$$
Likewise, $\phi(x_{i-1})$ and $\phi(x_i) = \phi(x_{i+1})$ commute, so $\phi(x_i) = \phi(x_{i-1})$, and by downward induction,
$$\phi(x_i) = \phi(x_{i-1}) = \cdots = \phi(x_1).$$
\end{proof}

\begin{cor}
\label{c:center}
If $\phi(x_i) \in Z(G)$ for some $i$, then $\phi$ is constant.
\end{cor}

\begin{cor}
\label{c:equal}
If $\phi(x_i) = \phi(x_{i+1})$ for some $i$, then $\phi$ is constant.
\end{cor}

If $j\ge i$, we use the notation $X_{[i,j]}$ for the product $x_i x_{i+1}\cdots x_j$; if $j < i$, we define
$X_{[i,j]}$ to be the identity.

\begin{lemma}
\label{l:trans}
For $k\ge 3$ and $1\le i\le k-2$, we have
$$X_{[1,k]} x_i X_{[1,k]}^{-1} = x_{i+1}.$$
\end{lemma}

\begin{proof}
\begin{align*}
X_{[1,k]} x_i X_{[1,k]} ^{-1} &= X_{[1,i-1]} x_i x_{i+1} X_{[i+2,k]} x_i X_{[i+2,k]}^{-1} x_{i+1}^{-1} x_i^{-1}
X_{[1,i-1]}^{-1} \\
&= X_{[1,i-1]} x_i x_{i+1} x_i x_{i+1}^{-1} x_i ^{-1} X_{[1,i-1]}^{-1} \\
&= X_{[1,i-1]}  x_{i+1} X_{[1,i-1]} ^{-1} \\
& = x_{i+1}.
\end{align*}
\end{proof}

\begin{lemma}
\label{l:pairs}
If $1\le i,j,k,l\le n-1$, $|i-j|\ge 2$, $|k-l|\ge 2$, then there exists $z = z_{i,j,k,l}\in B_n$ such that
$$z x_i z^{-1} = x_k,\ zx_jz^{-1} = z_l.$$
\end{lemma}

\begin{proof}
First we assume $i < j$ and $k<l$.
By Lemma~\ref{l:trans}, without loss of generality we may assume $j=l=n-1$.
As $X_{[1,n-3]}$ commutes with $x_{n-1}$, the ordered pair $(x_i,x_{n-1})$
can be conjugated to $(x_{i+1},x_{n-1})$ as long as $1\le i\le n-4$.  By induction on $i$,
all the $(x_i,x_{n-1})$ with $i\le n-3$ are conjugate.

To treat the case that $i > j$ or $k > l$, it suffices to prove that $(x_1,x_3)$ can be conjugated to
$(x_3,x_1)$.  Letting
$$y = x_1 x_2 x_3 x_1 x_2 x_1 = x_1 x_2 x_1 x_3 x_2 x_1,$$
we have
\begin{align*}
y x_1 &= x_1 x_2 x_3 x_1 x_2 x_1 x_1 = x_1 x_2 x_3 x_2 x_1 x_2 x_1 = x_1 x_3 x_2 x_3 x_1 x_2 x_1
= x_3 y; \\
y x_3 &= x_1 x_2 x_1 x_3 x_2 x_3 x_1 =  x_1 x_2 x_1 x_2 x_3 x_2 x_1 =  x_1 x_1 x_2 x_1 x_3 x_2 x_1
 = x_1 y.
\end{align*}
\end{proof}

Now let $0\to A\to G\to H\to 0$ be a central extension.  We write $[h_1,h_2]^\sim$ for the commutator
$g_1 g_2 g_1^{-1} g_2^{-1}\in G$, where $g_i$ is any element mapping to $h_i$.  As the extension is central, this is well-defined.

\begin{lemma}
\label{l:cmodc}
If $0\to A\to G\overset{\pi}{\to} H\to 0$ is a central extension and $\phi\colon B_n\to G$
is a homomorphism such that $\pi\circ\phi$ is constant, then $\phi$ is constant.
\end{lemma}

\begin{proof}
Any two elements of $G$ which map to the same element of $H$ must commute.  The lemma therefore follows from Proposition~\ref{p:consecutive}.
\end{proof}

\begin{prop}
\label{p:lift}
If $0\to A\to G\overset{\pi}{\to} H\to 0$ is a central extension and $\phi\colon B_n\to H$ is a homomorphism such that $[\phi(x_i),\phi(x_j)]^\sim = 1$ for some $i,j$ with $|i-j|\ge 2$, then
$\phi$ lifts to a homomorphism $\tilde\phi\colon B_n\to G$.
\end{prop}

\begin{proof}
As $[\ ]^\sim$ respects conjugation, Lemma~\ref{l:pairs} implies
$$[\phi(x_i),\phi(x_j)]^\sim = 1$$
for all $i,j$ with $|i-j|\ge 2$.  Fix an element $\tilde x_1\in G$ with $\pi(\tilde x_1) = \phi(x_1)$
and an element $\tilde y\in G$ with $\pi(\tilde y) = \phi(X_{[1,n-1]})$.  By Lemma~\ref{l:trans},
$$\pi(\tilde y^k \tilde x_1\tilde y^{-k}) = \phi(x_{k+1}),\ k=0,1,\ldots,n-2.$$
Let
$$g_i = \tilde y^{i-1} \tilde x_1 \tilde y^{1-i}.$$
Thus $g_i$ and $g_j$ commute when $|i-j|\neq 1$, and
the elements
$$a_i := g_i g_{i+1} g_i g_{i+1}^{-1} g_i^{-1} g_{i+1}^{-1}$$
are all conjugate in $G$ and lie in $A$.  Thus, they all coincide; denoting this common element $a$,
and setting $\tilde x_i = a^i g_i$, we have $\pi(\tilde x_i) = \phi(x_i)$, and the $\tilde x_i$ satisfy the relations (\ref{braid}) and (\ref{comm}).  Defining a homomorphism $\tilde\phi$ by the equations
$\tilde\phi(x_i) = \tilde x_i$, we see that $\tilde \phi$ is a lift of $\phi$.
\end{proof}

\begin{prop}
\label{p:solv}
If $n\ge 6$, then every homomorphism from $B_n$ to a solvable group $G$ is constant.
\end{prop}

\begin{proof}
We use induction on the length of the derived series.  The proposition follows immediately from Corollary~\ref{c:center} when
$G$ is abelian, so without loss of generality we may assume that the last non-trivial term $A$ in the derived series of $G$ is a proper subgroup of $G$.  By the induction hypothesis, any homomorphism $B_n\to G/A$ is constant.
We therefore choose an element $g\in G$ and a sequence $a_1,\ldots,a_{n-1}\in Z$ such that
$\phi(x_i) = a_i g$ for $i=1,\ldots,n-1$.  Writing $a^g$ for $gag^{-1}$, we have
$$a_i a_j^g g^2 = \phi(x_i x_j) = \phi(x_j x_i) = a_j a_i^g g^2$$
and therefore
$$a_i^{-1} a_i^g = a_j^{-1} a_j^g$$
whenever $|i-j|\ge 2$.  The graph on the vertex set $\{1,2,\ldots,n-1\}$ defined by the relation $|i-j|\ge 2$ is connected for $n\ge 6$.  Thus,
$$a_1^{-1} a_1^g = \cdots = a_{n-1}^{-1} a_{n-1}^g = a$$
for some $a\in A$.  The braid relation (\ref{braid}) for $j=i+1$ implies
$$a^3 a_i^2 a_{i+1} = a_i a_{i+1}^g a_i^{g^2} = a_{i+1} a_i^g a_{i+1}^{g^2} = a^3 a_i a_{i+1}^2,$$
so $a_1=\cdots=a_{n-1}$, and $\phi$ is constant as claimed.
\end{proof}

\begin{thm}
\label{t:rootbound}
If $\cG$ is a linear algebraic group over a field $K$ with solvable component group,
and $n \ge 6\sqrt{\dim \cG}+3$, then every homomorphism $B_n\to \cG(K)$ is constant.
\end{thm}

\begin{proof}
We assume without loss of generality that $K$ is algebraically closed.
We prove by induction on $r$ that the theorem is true whenever $2\sqrt{\dim \cG}\le r$ and $n\ge 3r$
the base case $r=1$ being trivial.
If $r\ge 2$, then $n\ge 6$, so the composition of $\phi\colon B_n\to \cG(K)$ with the quotient map
$\cG(K)\to \cG(K)/\cG^\circ(K)$ is constant.
We may therefore assume that $\cG/\cG^\circ$ is cyclic.
We may also assume that the theorem is true in dimension $<\dim \cG$.
If $\cU$ denotes the unipotent radical of $\cG^\circ$, then $\cU$ is a normal algebraic subgroup of $\cG$.
If the composition homomorphism $B_n\to (\cG/\cU)(K)$ is constant, then $B_n$ maps to a solvable subgroup of $\cG(K)$, namely, an extension of the (cyclic) image of this homomorphism by $\cU(K)$.
By Proposition~\ref{p:solv}, this implies that $\phi$ is constant.  Without loss of generality, therefore, we may assume that $\cG$ is reductive.  Likewise, composing $\phi$ with the quotient of $\cG$ by the center of $\cG^\circ$, we may assume without loss of generality that $\cG^\circ$ is adjoint semisimple.

If there exist positive dimensional normal subgroups $\cN_1,\ldots,\cN_t$  of $\cG$ such that
$\cN_1(K)\cap\cdots\cap \cN_t(K) = \{1\}$, then the compositions of $\phi$ with the projections
$\cG(K)\to (\cG/\cN_i)(K)$ are all constant, and therefore $\phi$ is constant.
If $\cG^\circ$ has at least two non-isomorphic simple factors, then the product of all factors of any one type is a proper normal subgroup of $\cG$.  We may therefore assume that $\cG^\circ \cong \cH^k$ for some positive integer $k$ and some (adjoint) simple algebraic group $\cH$.  Moreover, conjugation by
a generator of $\cG/\cH^k$ induces a well-defined outer automorphism of $\cH^k$ and therefore a permutation $\sigma$ of the factors, which are the minimal non-trivial normal subgroups of $\cH^k$.
Without loss of generality we may assume that this permutation is a $k$-cycle, since otherwise, each orbit of $\sigma$ determines a product of factors $\cH$ which is a normal subgroup of $\cG$.

Let $x = \phi(x_{n-2})^{-1} \phi(x_{n-1})$, and let $B_{n-3}$ denote the subgroup of $B_n$ generated
by $x_1,\ldots,x_{n-4}$.  Thus, $x$ lies in $\cH(K)^k$ and $\phi(B_{n-3})$ lies in the centralizer
of $x$ in $\cG(K)$.  If $x$ is the identity, then $\phi$ is constant by Corollary~\ref{c:equal}.  Therefore, the image of $x$ in one of the factors of $\cH(K)^k$ is non-trivial, so without loss of generality we may assume that the centralizer $\cZ_x\subset \cG$ satisfies
$$\cZ_x \cap \cH^k \subset \cH_1\times \cH^{k-1},$$
where $\cH_1$ is the centralizer of a non-trivial element $x_1$ of $\cH$.
Since every element $x_1\in \cH(K)$ has a Jordan decomposition, and the centralizer of
$x_1$ is contained both in the centralizer of its unipotent part and in that of its semisimple part, without
loss of generality we may assume that $\cH_1$ is the centralizer of either a semisimple
or a unipotent element.
We have seen that $\phi(B_{n-3})$ lies in $\cZ_x(K)$.  Suppose $\cZ_x$ fails to meet every component
of $\cG$.  Then the factors of $\cH^k$ form $t\ge 2$ orbits of equal cardinality
under conjugation by $\cZ_x$.
Let $\cN_i$ denote the algebraic subgroup of elements in $\cZ_x\cap \cH^k$
which are trivial on all factors
$\cH$ belonging to the $i$th orbit of this action, it suffices to prove that the composition $\phi_i$ of
$\phi|_{B_{n-3}}$ with the projection $\cZ_x(K)\to (\cZ_x / \cN_i)(K)$ is constant for each $i$.  However,
$$\dim \cZ_x/\cN_i \le \frac{\dim \cH^k}{t} = \frac{\dim \cG}{t} \le r^2/8 \le ((r-1)^2)/4$$
since $\cG$ non-trivial and semisimple implies $r\ge 4$.
Thus, the induction hypothesis implies that each $\phi_i$ is constant and therefore that $\phi$ is constant.

This leaves the case that $t=1$.  In this case, $\cZ_x\cap \cH^k$ is contained in a product
$\cH_1\times\cdots\times \cH_k$ where each $\cH_i$ is isomorphic to $\cH_1$.
If $\cR_i$ denotes the radical of $\cH_i^\circ$, then $\cR_1\times\cdots\times \cR_k$ is a normal subgroup
of $\cZ_x$, and it suffices to prove that every homomorphism
$$B_n\to (\cZ_x/(\cR_1\times\cdots\times \cR_k))(K)$$
is constant.  The component group of $\cZ_x$ is isomorphic to an extension of a cyclic group
by $(\cH_1/\cH_1^\circ)^k$.  If $\cH_1$ is the centralizer of a semisimple element, the group of
components is commutative.  This is an immediate consequence of the
theorem of Springer-Steinberg asserting that
the centralizer of a semisimple element in a simply-connected semisimple group is connected.
If $\cH_1$ is the centralizer of a unipotent element, the component group $\cH_1/\cH_1^\circ$ is either
solvable or isomorphic to $S_5$, by a theorem of Alexeevski \cite{A} and Mizuno \cite{M}.  Except in the last case, which can only occur if $\cH\cong E_8$, the
component group of $\cZ_x/(\cR_1\times\cdots\times \cR_k)$ is solvable, so the induction hypothesis applies.
Every automorphism of $S_5\cong \PGL_2(\F_5)$ is inner and therefore extends to an automorphism of
the algebraic group $\PGL_2$ in characteristic $5$.  In the case of $S_5$, $n\ge 6\sqrt{248k}+3$,
so $n-3\ge 6\sqrt{3k}+3$.  Thus, the induction hypothesis implies that any homomorphism
from $B_{n-3}$ to an extension of $\Z/k\Z$ by $S_5^k$ is constant, and replacing $\cZ_x$ with a suitable open subgroup, we may again assume that $\cZ_x/\cZ_x^\circ$ is solvable.

Finally, applying the classification of maximal subgroups of simple algebraic groups
\cite{Se1,Se2}, we see that
$$2\sqrt{\dim(\cH_1/\cR_1)} < 2\sqrt{\dim \cH}-1.$$
The theorem now follows by induction.

\end{proof}

A variant of this argument gives the following:

\begin{thm}
\label{t:sporadic}
If $H$ is a finite simple group, $G := H^k\rtimes C$, where $C$ is solvable, and
$n \ge 3\log_2 k|H|$, then every homomorphism $B_n\to G$ is constant.
\end{thm}

\begin{proof}
If $|G|\le 5$, then $|G|$ is abelian, so without loss of generality we may assume $n\ge 6$.
Following the proof of Theorem~\ref{t:rootbound}, we may therefore assume
that $C\cong \Z/k\Z$ cyclically permutes the factors of $H^k$.  Given $\phi\colon B_n\to G$,
we let $x = \phi(x_{n-2})^{-1} \phi(x_{n-1})$, and let $Z_x$ denote the centralizer of $x$ in $G$.
As $x\in H^k$, we fix a factor on which the projection of $x$ is non-trivial, and let $H_1$ denote
the centralizer of this projection of $x$ in this factor.  If $Z_x$ does not map onto $C$, then there exist normal subgroups
$N_1,\ldots,N_t$ of $Z_x$ with $N_1\cap\cdots\cap N_t = \{1\}$ such that each $Z_x / N_i$
is contained in a group of the form $H^d\rtimes \Z/d\Z$, where $d$ is a proper divisor of
$k$, so $\log_2 d \le \log_2 k-1$.  By the induction hypothesis every homomorphism from
$B_{n-3}$ to $Z_x/N_i$ is constant, so $\phi|_{B_{n-3}}$ is constant, and as $n\ge 6$,
this means that $\phi(x_1) = \phi(x_2)$ and therefore that $\phi$ is constant.

If $Z_x$ does map onto $C$, then $Z_x$ is a subgroup of a group isomorphic to to
$H_1^k\rtimes C$, where $H_1$ is the centralizer of a non-trivial element of $H$.
As $\log_2 |H_1| \le \log_2 |H|$, the induction hypothesis implies that
$\phi|_{B_{n-3}}$ is constant and therefore that $\phi$ is constant.

\end{proof}

\section{Representations of Linearly Bounded Degree}

In this section, we examine the possible degrees of low-dimensional unitary representations of a braid group $B_n$.  The complex irreducible representations of degree $\le n$ of $B_n$ have been completely described \cite{FLSV,Sy}.  The constant representations have degree $1$, and the non-constant representations in this range have degree $n-2$, $n-1$, or $n$.  Sysoeva \cite{Sy} has announced that there are no irreducible representations of degree $n+1$ for $n$ sufficiently large, and has conjectured that such a statement holds for degree
$n+k$ as well.

In this section, we consider the irreducible unitary representations of $B_n$ of degree $\le ln$
where $l$ is a fixed integer and $n$ is sufficiently large in terms of $l$.

We say that a sequence $d_0,d_1,d_2,\ldots$
is \emph{weakly convex} if the sequence of differences $d_1-d_2,d_2-d_3,\ldots$ is non-increasing.

\begin{lemma}
\label{l:ineq}
If $d_0,d_1,\ldots$ is a weakly convex sequence and $i<j<k$, then there exists an integer $s$ such that
$$\frac{d_j-d_i}{j-i}\le s\le \frac{d_k-d_j}{k-j}.$$
\end{lemma}

\begin{proof}
Setting $s = d_{j+1}-d_j$, the lemma follows immediately.
\end{proof}

\begin{lemma}
\label{l:convex}
Let $V$ be a finite-dimensional vector space, $W\subset V$ a subspace, and $T\colon V\to V$ an
invertible linear transformation.  The sequence $d_0,d_1,d_2,\ldots$ defined by $d_0:=\dim V$ and
$$d_k := \dim W\cap T(W)\cap T^2(W)\cap\cdots\cap T^{k-1}(W),\ k\ge 1$$
is weakly convex.
\end{lemma}

\begin{proof}
Define $W_0=V$, and
$$W_k := W\cap T(W)\cap T^2(W)\cap\cdots\cap T^{k-1}(W),\ k\ge 1.$$
Then
$$d_k - d_{k+1} = \dim W_k - \dim W_{k+1} = \dim W_k/W_{k+1}$$
As $T^{-1}$ maps to $W_{k+1}$ to $W_k$ and $W_{k+2}$ to $W_{k+1}$, it induces a map
$$W_{k+1}/W_{k+2}\to W_k/W_{k+1}.$$
As
$$W_{k+1}\cap T(W_{k+1}) = W_{k+2},$$
this linear transformation is injective, so
$$d_k - d_{k+1}\ge d_{k+1}-d_{k+2}.$$
\end{proof}

We apply this lemma in the following way.  Let $V$ be a finite-dimensional complex vector space endowed with a Hermitian inner product,
and $\phi\colon B_n\to U(V)$ an irreducible unitary representation.
For each $\lambda\in\C$, we define $W=W^\lambda$ to be the $\lambda$-eigenspace of
$\phi(x_1)$.  By Lemma~\ref{l:trans} there exists $y\in B_n$ such that $yx_i y^{-1} = x_{i+1}$
for $1\le i\le n-2$.  We set $T=\phi(y)$.  Now, $w\in W^\lambda$, if and only if
$$(\phi(x_1)-\lambda)(w) = 0.$$
For any $k$, this is equivalent to
$$(\phi(y^{1-k}x_k y^{k-1})-\lambda)(w)=0,$$
or to
$$(\phi(x_k)-\lambda)(\phi(y^{k-1})(w)) = 0.$$
Thus, the $\lambda$-eigenspace of $\phi(x_k)$ is $T^{k-1}(W^\lambda)$.

We say that an irreducible representation $\phi\colon B_n\to U(m)$ is of \emph{level $k$} if
one of the following is true:
\begin{enumerate}
\item $k=0$ and $m=1$.
\item $k\ge 1$ and $kn-(k^2+3k-2)\le m\le kn$.
\end{enumerate}

\begin{thm}
\label{t:degrees}
For every integer $l\ge 1$ and every integer $n$ sufficiently large in terms of $l$,
every irreducible unitary representation of the braid group $B_n$ of degree $\le ln$
is of some (unique) level $k\le l$.
\end{thm}

\begin{proof}
As
$$(k-1)n < kn-(k^2+3k-2)$$
when $n$ is sufficiently large, uniqueness is clear.
For existence, we use induction on $l$, the $l=1$ case being known \cite{FLSV}.
For given $l\ge 2$, let $\phi\colon B_n\to \Aut(V)$ be an irreducible unitary representation of degree
$\le ln$.  We may therefore assume that
\begin{equation}
\label{range}
(l-1)n+1\le \dim V \le ln - (l^2+3l-1)
\end{equation}
We write $B_{n-1}$ and $B_{n-2}$ for the subgroups of $B_n$ generated
by $x_i$ with $1\le i\le n-2$ and $1\le i\le n-3$ respectively.

For each eigenvalue $\mu$ of $\phi(x_{n-1})$, let $X^\mu$ denote the $\mu$-eigenspace.
As $B_{n-2}$ commutes with $x_{n-1}$, $\phi(B_{n-2})$ acts on $X^\mu$.
We say that $X^\mu$ \emph{splits} if it is a direct sum of constant representations of $B_{n-2}$.
A sufficient condition that $X^\mu$ splits is
$$\dim X^\mu \le n-5,$$
as the minimum degree of a non-constant representation of $B_{n-2}$ is $n-4$.  Let $X$ denote the direct sum of
all irreducible $1$-dimension factors of $B_{n-2}$ in $V$, so $X$ contains the sum of
all split $X^\mu$.
Let  $\lambda_1,\ldots,\lambda_r$ be the constants appearing in $X$ regarded as a $B_{n-2}$-representation, and  let $W^{\lambda_i}$ denote the
$\lambda_i$-eigenspace of $\phi(x_1)$ on $V$, which of course contains the
$\lambda_i$-eigenspace of $\phi(x_1)$ on $X$.
Thus $W_j^{\lambda_i}$ is the intersection of the $\lambda_i$-eigenspaces of
$\phi(x_1),\ldots,\phi(x_j)$.  As $W_{n-1}^{\lambda_i} = \{0\}$, Lemma~\ref{l:convex} implies
$$\dim W_j^{\lambda_i} \ge \frac{n-1-j}{2} \dim W_{n-3}^{\lambda_i},$$
for $1\le j\le n-3$.  If $\dim X\ge 2l+1$,
$$ln \ge \sum_i \dim W_1^{\lambda_i} \ge \frac{n-2}{2}\dim X\ge  ln + (n/2-2l-1),$$
Assuming $n > 4l+2$, we may therefore conclude that $\dim X\le 2l$.

We consider first the case that there are at least two different eigenvalues $\mu_i$ such that
$X^{\mu_i}$ does not split.  For each $\mu\in\{\mu_1,\ldots,\mu_r\}$, let $X^\mu_{ns}$
denote the orthogonal
complement in $X^\mu$ of the direct sum of all constant representations of $B_{n-2}$.
Then
\begin{align*}
\dim V - \dim X^\mu_{ns} &\le ln - (l^2+3l-1) - (n-4) \\
&= (l-1)(n-2) - (l^2+l-3) \\
&< (l-1)(n-2),
\end{align*}
so $\bigoplus_{\mu_i\neq\mu} X^{\mu_i}_{ns}$
satisfies the induction hypothesis for representations of $B_{n-2}$, and the same is true of each
irreducible factor of each $X^{\mu_i}_{ns}$.  Each irreducible factor of $X^{\mu_i}_{ns}$ therefore has a level.
Letting $k_1,k_2,\ldots,k_s\ge 1$ denote the sequence of levels, we have
$$\dim V = \dim X + \sum_{i=1}^s \dim X^{\mu_i}_{ns},$$
so
$$(k_1+\cdots+k_s)(n-2) - \sum_{i=1}^s (k_i^2+3k_i-2) \le \dim V \le 2l + (k_1+\cdots+k_s)(n-2)$$
For $n$ sufficiently large in terms of $l$, this, together with (\ref{range}) implies
$k_1+\cdots+k_s = l$.   As $x^2+3x-2$ is convex, for any fixed values of $s\ge 2$
and $l$, the sum of $k_i^2+3k_i-2$ is minimized, subject to the constraints $k_i\ge 1$ and
$k_1+\cdots+k_s=l$, when all but one value of $k_i$ is $1$.
As the difference between values of $x^2+3x-2$ for consecutive positive integers
exceeds the value at $x=1$, if $s$ is constrained to be greater than $1$ but otherwise
can be chosen freely, the sum of $k_i^2+3k_i-2$ is maximized when $s=2$.  Thus,
\begin{align*}
\dim V &\ge (k_1+\cdots+k_s)(n-2) - \sum_{i=1}^s (k_i^2+3k_i-2) \\
&\ge ln - 2l - (l-1)^2 - 3(l-1)+2 - 2 \\
& = ln - (l^2+3l-2).
\end{align*}

This leaves the case that there exists a unique $\mu$ such that $X^\mu_{ns}$ is not zero.
Let $X^\mu_i$ denote the intersection of the $\mu$-eigenspaces of
$x_{n-1},x_{n-2},\ldots,x_{n-i}$.  By Lemma~\ref{l:convex}, applying Lemma~\ref{l:ineq}
for $0<i<j$,
$$\dim X^\mu_i - \dim X^\mu_j \le (j-i) \biggl\lfloor\frac{\dim V - \dim X^\mu_i}i\biggr\rfloor.$$
If
$$\dim V - \dim X^\mu_l < l^2,$$
then setting $j=n-1$ and $i=l$, we have
\begin{align*}
\dim V &\le l^2-1+\dim X_l^\mu\le l^2-1 + \dim X_l^\mu - \dim X_{n-1}^\mu \\
&\le l^2-1 +(n-l-1)\biggl\lfloor\frac{l^2-1}l\biggr\rfloor \\
&=(l-1)n,
\end{align*}
which for $n$ sufficiently large is inconsistent with (\ref{range}).
On the other hand,
$$\dim V - \dim X^\mu \le 2l,$$
so
$$\dim V - \dim X^\mu_l \le 2l^2.$$
Assuming that $2l^2\le n-l-6$, this implies that the orthogonal complement of
$X^\mu_l$ is a split representation of $B_{n-l-1}$, the subgroup of $B_n$ generated by
$x_1,\ldots,x_{n-l-2}$.

Let $\lambda_i$ denote the eigenvalues of this representation.
We have
$$\sum_i \dim W^{\lambda_i}_{n-l-2} \ge l^2.$$
On the other hand, $\dim W^{\lambda_i}_{n-1} = 0$.  By Lemmas \ref{l:ineq} and \ref{l:convex},
$$\dim W^{\lambda_i}_1 - \dim W^{\lambda_i}_{n-l-2}
\ge (n-l-3)\biggl\lceil\frac{\dim W^{\lambda_i}_{n-l-2}}{l+1}\biggr\rceil.$$
As  $\lceil x/(l+1)\rceil$ is superadditive in $x$ and $\lceil l^2/(l+1)\rceil = l$,
\begin{align*}
\sum_i \dim W^{\lambda_i}_1&\ge \sum_i \dim W^{\lambda_i}_{n-l-2}
	+ (n-l-3)\biggl\lceil\frac{\dim W^{\lambda_i}_{n-l-2}}{l+1}\biggr\rceil \\
&\ge l^2 + (n-l-3)l = nl - 3l,
\end{align*}
contrary to (\ref{range}).
\end{proof}

In particular by the proof of Theorem~\ref{t:degrees} we see that $B_n$ has no irreducible $(n+1)$-dimensional unitary representations for $n\geq 16$.  The actual lower bound is at least $8$ as $B_7$ has irreducible $8$-dimensional unitary representations (factoring over the Hecke algebra $H(i,7)$, see \cite{J1}).

Theorem~\ref{t:degrees} can be extended to projective unitary representations.  In fact, we have the following proposition:

\begin{prop}
\label{p:proj}
Every irreducible projective unitary representation of $B_n$ of degree $d\le 2^{n/6}$ lifts to a linear representation of $B_n$.
\end{prop}

\begin{proof}
The proposition is trivial for $n\le 5$.  We may therefore assume $n\ge 6$.
Thus there exists a sequence $a_1 < \cdots < a_{2m}$
of positive odd integers less than $n$, with $m\ge n/6$.  Let $y_i = x_{a_i}$.  The generators
$y_i$ commute with one another.  The central extension
$$0\to U(1)\to U(d)\overset{\pi}{\to} \PSU(d)\to 0$$
defines a commutator map $[\ ]^\sim$.  By Lemma~\ref{l:pairs}, $[\phi(x_i),\phi(x_j)]^\sim$
is independent of the pair $(i,j)$ provided $|i-j|\ge 2$.  It is therefore symmetric as well as antisymmetric and consequently takes values $\pm 1$.  If $[\phi(x_i),\phi(x_j)]^\sim = 1$ for some
(and therefore all) $(i,j)$ with $|i-j|\ge 2$, then by Proposition~\ref{l:cmodc}, $\phi$ lifts to
a homomorphism to $U(m)$.

We therefore assume that $[\phi(y_i),\phi(y_j)]^\sim = -1$ for all $i\neq j$.  Let
$$a_i = y_1 y_2\cdots y_{2i-1},\ b_i = y_1 y_2 \cdots y_{2i-2} y_{2i}.$$
Then
$$[a_i,a_j]^\sim = [b_i,b_j]^\sim = 1,\ [a_i,b_j]^\sim = (-1)^{\delta_{ij}}.$$
Let
$$G_i:= \pi^{-1}(\phi(\langle a_i,b_i\rangle)).$$
Clearly, the restriction of the standard representation of $U(m)$ to $G_i$
has no $1$-dimensional components.
The subgroups $G_1,\ldots,G_m\subset U(d)$ commute in pairs and give rise to a homomorphism
$G_1\times\cdots G_m\to U(d)$.  The restriction of the standard representation of $U(m)$
to this product decomposes as a sum of irreducible representations of
$G_1\times\cdots\times G_m$, each of which is an external tensor product of representations
of the $G_i$, each of degree $>1$.  Therefore, $d\ge 2^m$.
\end{proof}

\section{Representations of Exponentially Bounded Degree}

In this section we fix a constant $c$ and consider non-constant unitary representations of $B_n$,
$n\ge 6$, of degree $d \le c^n$ with finite image.   We are interested in the behavior of
$G:=\rho(B_n)$.  By Proposition~\ref{p:solv}, $G$ cannot be solvable.

\begin{defn}
A finite group $G$ is \emph{almost characteristically simple} if there exists a
(non-abelian) finite simple group $H$ and a positive integer $k$ such that $H^k<G<\Aut(H^k)$.  We say $G$ is \emph{of permutation type} if $H$ is isomorphic to the alternating group $A_n$ for some
$n\ge 5$.
\end{defn}

\begin{prop}
If $G$ is any finite group which is not solvable and $K$ is maximal among normal subgroups of $G$ such that
$G/K$ is not solvable, then $G/K$ is almost characteristically simple.
\end{prop}

\begin{proof}
Replacing $G$ by $G/K$, we may assume that $G$ is not solvable but every non-trivial quotient group
of $G$ is.
In particular, $G$ has no non-trivial normal abelian subgroup.  Let $H^k$ denote a characteristically simple normal subgroup of $G$.  Thus, $H$ is non-abelian.  Let $K$ denote the centralizer
of $H^k$ in $G$.  Then $K$ is a normal subgroup.  The quotient $G/K$ is not solvable because it
contains the homomorphic image of $H^k$, which is isomorphic to $H^k$ itself since the
center of $H$ is trivial.
It follows that $G/K$ is trivial and therefore that the action of $G$ on $H^k$ by conjugation is faithful, i.e.,
$H^k \subset G\subset \Aut(H^k)$.
\end{proof}

\begin{defn}
If $G$ is a finite group which is not solvable, a \emph{minimal quotient} is any group of the form $G/K$ where $K$ is maximal among normal subgroups of $G$ such that $G/K$ is not solvable.
\end{defn}

\begin{defn}
A finite group is of \emph{classical type of rank $r$} if it is a finite simple group of the form $A_r(q)$, $^2A_r(q)$,$B_r(q)$, $C_r(q)$, $D_r(q)$, or $^2 D_r(q)$
\end{defn}

Roughly speaking, a finite simple group is of classical type if
it is a linear, unitary, orthogonal, or symplectic group over a finite field.

\begin{thm}
\label{t:images}
For every constant $c$ there exist positive constants $A$, $B$, $K$, $N$, and $Q$ such that for all
$n>N$ and all $\rho\colon B_n\to U(d)$ with $d\le c^n$ and finite image $G$,
 every minimal quotient of $G$
is either of permutation type or of the form $H^k\rtimes \Z/m\Z$, where $H$ is a finite simple group of
classical type of rank $r$.  In the latter case, $1\le k\le K$, $2\le q\le Q$, and $An\le r\le Bn$.
\end{thm}

\begin{proof}
A minimal quotient is of the form $H^k\rtimes C$, where $C$ is solvable and $H$ is simple.   By hypothesis, $H$ is not an alternating group.  By Theorem~\ref{t:sporadic}, if $n$ is sufficiently large,
then $|H|$ can be taken to be as large as we wish; in particular, we exclude that case that $H$ is sporadic.  By Theorem~\ref{t:rootbound}, if $n$ is sufficiently large and $H$ is of Lie type, the dimension of the underlying simple algebraic group must be $>\epsilon n^2$ for some absolute constant
$\epsilon>0$, so the rank $r$ of the group must be greater than $A n$
for some absolute constant $A>0$.
Thus, we may assume that $H$ is a perfect group whose universal central extension is $\cH(\F)$,
where $\cH$ is a simply connected semisimple algebraic group over $\F$ which is absolutely simple modulo its center and of rank $r\ge 9$.

Let $G_0$ denote the inverse image of $H^k\subset H^k\rtimes C$ in $G$.
We have a short exact sequence
$$0\to J\to G_0\to H^k\to 0,$$
which we pull back to a short exact sequence
\begin{equation}
\label{split}
0\to J\to \tilde G_0\to \cH(\F)^k\to 0.
\end{equation}
As $\tilde G_0$ is a central extension of $G_0$, the faithful representation $G_0\to U(d)$
gives rise to an almost faithful $d$-dimensional representation of $\tilde G_0$.  We claim that this implies that $d$ is greater than or equal to the degree of the minimal non-trivial representation
of $\cH(\F)$.  Let $X\subset \Hom(Z(J),\C^\times)$ denote
the set of characters obtained by restricting $\tilde G_0\to U(d)$ to the abelian group $Z(J)$.
Thus $\cH(\F)^k$ acts on $X$.  If this action is non-trivial, then the permutation representation of
$\cH(\F)^k$ acting on $X$ is non-trivial and therefore contains a non-trivial factor.  The minimal degree for a non-trivial representation of $\cH(\F)^k$ is the same as that for $\cH(\F)$.  We may therefore assume
that $\cH(\F)^k$ acts trivially on $X$.  This implies that the action of $\cH(\F)^k$ on $Z(J)$
preserves both $Z(\tilde G_0)\subset Z(J)$ and $Z(\tilde G_0) / Z(J)$ pointwise.  As
$\cH(\F)^k$ is perfect, any action of this group on an abelian group which fixes a subgroup and quotient group pointwise is trivial.  It follows that
$Z(J)$ lies in the center of $\tilde G_0$.  The non-abelian cohomology class which determines whether
(\ref{split}) splits lies in $H^2(\cH(\F)^k,J)$, which is a principal homogeneous space of
$H^2(\cH(K)^k,Z(J))$.  The latter is trivial since $\cH(K)^k$ is centrally closed.  Therefore,
$G_0$ contains a subgroup isomorphic to $\cH(\F)^k$, and restricting $V$ to this subgroup, we see that our claim holds.

The Seitz-Landazuri bound \cite{LS} on the minimal degree projective representations of finite simple groups of Lie types now implies that $q^{k r/n}$ is bounded in terms of $c$.  Given that
$r/n > A$, this gives upper bounds $Q$ and $K$ for $q$ and $k$, and given that $q\ge 2$, $k\ge 1$, this gives an upper bound $B$ for $r/n$.

\end{proof}

We remark that the theorem can be extended in two ways without essentially modifying the proof.
On the one hand, we need not assume that the representation $V$ is unitary.  On the other hand, if $V$ is unitary, we need not assume that $\rho(B_n)$ is finite; we can take the closure of the image, obtain a compact Lie group, and characterize the \emph{group of components} of this Lie group without assuming that the identity component is trivial.

\section{An Application}

We would like to describe a general setting in which one obtains sequences of unitary representations of the braid group of exponentially bounded degree.  Let $\CC$ be any unitary premodular =(ribbon fusion) category (see \cite[Chapter II.5]{T}).  In particular this means that $\CC$ is semisimple with finitely many (isomorphism classes of) simple objects $\{X_0,\cdots,X_r\}$ and the morphism spaces are finite dimensional $\C$-vector spaces.  Moreover, such a category is equipped with a conjugation and a positive definite Hermitian form with respect to which each $\End(X^{\ot n})$ is a Hilbert space.  The braiding isomorphisms $c_{X,Y}:X\ot Y\cong Y\ot X$ induce unitary representations $\rho_n^X: B_n\rightarrow U(\End(X^{\ot n}))$ via:

$$\rho_n^X(\sigma_i)f=\Id_X^{\ot i-1}\ot c_{X,X}\ot \Id_X^{\ot n-i-1}\circ f$$
for any object $X$, where the $B_n$-invariance of the Hermitian form is included in the axioms.  By semisimplicity of $\End(X^{\ot n})$ the spaces $\Hom(X_j,X^{\ot n})$ for simple $X_j$ are equivalent to (potentially reducible) unitary $B_n$ subrepresentations of $\End(X^{\ot n})$.

We will show that $\dim\Hom(X_j,X^{\ot n})$ is exponentially bounded.
For simplicity of notation we assume that $X=X_i$ is a simple object and each object is isomorphic to its dual; the general case is essentially the same.
For each simple object $X_i$ we define a (symmetric) matrix $N_i$ whose $(j,k)$-entry is $\dim\Hom(X_k,X_i\ot X_j)$.  The matrices $N_i$, $0\leq i\leq r$ pairwise commute, and are clearly non-negative.  Let $d_i$ be the Perron-Frobenius eigenvalue of $N_i$, \emph{i.e.} the largest eigenvalue.  Setting $D=\max\{d_i\}$ we will show that  $\dim\Hom(X_j,X^{\ot n})\leq D^n$.  First observe that $d_i\geq 1$, since $|\lambda|\leq d_i$ for all other eigenvalues $\lambda$ and clearly $(N_i)^n\neq 0$ for all $n$.
It follows from the Perron-Frobenius Theorem that the vector $\mathbf{d}=(d_0,d_1,\cdots,d_r)^T$ is a strictly positive eigenvector with eigenvalue $d_i$ for each $N_i$, uniquely determined up to rescaling (one applies the Perron-Frobenius Theorem to the strictly positive matrix $M:=\sum_i N_i$, see e.g. \cite{ENO}).  Now denoting by $\mathbf{e}_i$ the $i$th standard basis vector for $\R^r$, we see that
$\dim(X_j,X_i^{\ot n})$ is the $j$th entry of $(N_i)^{n-1}\mathbf{e}_i$ which is less than or equal to the $j$th entry of $(N_i)^{n-1}\mathbf{d}=(d_i)^{n-1}\mathbf{d}$ which in turn is bounded by $D^{n}$.

There are two well-known constructions of unitary premodular categories.  The first is $\Rep(D^\omega G)$: the representation category of the twisted quantum double of a finite group $G$.  $D^\omega G$ is a semisimple $|G|^2$-dimensional quasi-triangular quasi-Hopf algebra (see \cite{BK}), and $\Rep(D^\omega G)$ is a modular category.  The braid group representations were studied in \cite{ERW} and found to have finite images.  In particular the image of
$\rho_n^H$ where $H=D^\omega G$ is the left regular representation of $D^\omega G$ is found to be a subgroup of the full monomial group $S_n\rtimes \Z_s^n$ for some $s$ and hence of permutation type.  Since any simple object appears as a subobject of $H$, it follows that all images are of permutation type.  The second set of examples come from representations of quantum groups at roots of unity (see e.g. \cite{R}) or, equivalently, from fixed level representations of affine Kac-Moody algebras.  Quantum groups of type $A_k$ at $4$th and $6$th roots of unity yield modular categories supporting braid group representations with finite images.  In fact, these representations factor over quotients of Hecke algebras $H(q,n)$ and are precisely those alluded to in the introduction.  Quantum groups of type $C_2$ at $10$th roots of unity also yield finite braid group images \cite{J3}, with images $\Sp(n-1,\F_5)$.  Here the object $X$ of interest has $d_X=\sqrt{5}$, and for $B_n$ with $n$ odd, $\dim\End(X^{\ot n})=(\sqrt{5})^{n-1}$ and is the metaplectic representation of $\Sp(n-1,\F_5)$ with two irreducible subrepresentations of dimension $\frac{(\sqrt{5})^{n-1}\pm 1}{2}$.  It appears that this can be generalized: there is evidence that quantum groups of type $B_k$ at $(4k+2)$th roots of unity and $D_k$ at $4k$th roots of unity support braid group representations with
finite symplectic groups as images.

\end{document}